\documentclass[10pt]{article}
\usepackage{amssymb}
\usepackage{amsthm}
\usepackage{graphicx}
\usepackage[shortcuts]{extdash}
\newtheorem{theorem}{Theorem}
\newtheorem{lemma}{Lemma}
\newtheorem{corollary}{Corollary}
\newtheorem{remark}{Remark}

\newtheorem{conjecture}{Conjecture}

\newtheorem{obs}{Observation}
\newcommand{\beq}{\begin{equation}}
\newcommand{\eeq}{\end{equation}}
\def\Frac#1#2{\frac{\displaystyle{#1}}{\displaystyle{#2}}}

\def\RE{\mathbb R}
 \title{Uniform relations between the Gauss-Legendre nodes and weights}
\author{
\'Oscar L\'opez Pouso\footnotemark[1]
 \and
Javier Segura\footnotemark[2]
\\
}

\begin{document}

\maketitle

\renewcommand{\thefootnote}{\fnsymbol{footnote}}

\footnotetext[1]{Department of
Applied Mathematics, Universidade de Santiago de Compostela and Centro de Investigaci\'on y Tecnolog\'{\i}a Matem\'atica 
de Galicia (CITMAga), 15782 Santiago de Compostela, Spain. E-mail: oscar.lopez@usc.es}
\footnotetext[2]{ Departamento de Matem\'aticas, Estad\'{\i}stica y
        Computaci\'on. Universidad de Cantabria, 39005 Santander, Spain. E-mail: javier.segura@unican.es}

\begin{abstract}
Four different relations between the Legendre nodes and weights are presented which, unlike the
circle and trapezoid theorems for Gauss-Legendre quadrature, hold uniformly in the whole interval $(-1,1)$. 
These properties are supported by strong asymptotic evidence.
The study of these results was originally motivated by the role some of them
 play in certain finite difference schemes used in the discretization of the angular Fokker-Planck
 diffusion operator.
\end{abstract}

{\bf Keywords:} Gauss-Legendre quadrature, asymptotic expansions, uniform relations, finite difference schemes, angular Fokker\=/Planck diffusion operator.

{\bf MSC2020:} 65D32, 33C45, 65D25, 65Z05.

\section{Introduction}

Given a quadrature rule $Q(f)$ for approximating integrals $I(f)$ in the interval $[-1,1]$
$$
I(f)=\displaystyle\int_{-1}^{1}f(x)dx\approx Q(f)=\displaystyle\sum_{i=1}^n w_i f(x_i),
$$
the formula is said to be of Gauss-Legendre type if it has the maximum possible degree of exactness, that
is, if $I(f)=Q(f)$ for $f$ any polynomial of degree less than $2n$. Gauss-Legendre quadrature is a particular
case of a more general set of quadrature rules, known as Gauss quadratures, 
and it is the most widely used of them. 
The fact that this rule has maximum degree of exactness implies that it is rapidly convergent as the number
of nodes increases, in particular for integrating analytic functions.

Gauss-Legendre nodes ($x_i$) and weights ($w_i$) are important quantities appearing in numerous applications, as they are not
only relevant in numerical integration but also in the approximation by barycentric interpolation at the Legendre nodes (see for instance the applications mentioned in \cite{Bogaert:2014:IFC} and
\cite{Gil:2019:NIG}) and in spectral methods. 
In the present paper, the study of the uniform relations between the Legendre nodes and weights
has been motivated
by the role these (so far unproven) relations play in certain finite difference schemes used in the field
of nuclear engineering, specifically for discretizing the angular
Fokker\=/Planck diffusion operator (see \cite{HA-LI,MO,Lopez:2023:AOD}). 
In this paper we provide strong asymptotic
 evidence for such uniform results.

In \cite{Davis:1961:SGT}, some asymptotic relations between the nodes and weights of Gauss quadratures
were explored, which hold for all the nodes in any fixed compact subinterval of $(-1,1)$; 
see also \cite{Gautschi:2006:TCT} for a more recent account of this type of results.
The
fact that these results hold in fixed compact intervals is relevant as this means that such results
may not be valid close to the endpoints of the interval of orthogonality.
In particular, for the Gauss-Legendre quadrature the nodes are known to cluster at the
ends of the interval $(-1,1)$ as the degree tends to infinity, and as a consequence
those results can fail (and in fact do fail) for the extreme nodes, as explained in section \ref{sectcir}.

In this paper we present four cases of asymptotic relations between nodes and weights of Gauss-Legendre quadrature
which hold uniformly for all the nodes in the interval of orthogonality and they do so with a nearly constant error term (tending to zero as the degree increases). Similar relations may exist
for more general quadratures, and they will be explored in a future paper.

The paper is organized as follows. In section \ref{prelim} we summarize known 
asymptotic results for the Legendre nodes and
weights, both considering expansions in terms of elementary functions and in terms of Bessel functions. In section \ref{sectcir}
the circle and trapezoid theorems, first presented in \cite{Davis:1961:SGT}, are revisited and it is discussed how these relations
 fail for the extreme nodes and how a uniform version of the circle theorem (which holds for all the nodes) 
 can be easily obtained. Uniform relations involving sums of weights are described in \ref{uniformr}, which can been
  understood as a uniform counterpart of the trapezoid theorems; three uniform relations are described in this section.
 Finally, section \ref{nuclear} describes the connection of these uniform relations with some finite difference schemes, 
 in particular for discretizing the angular
Fokker-Planck diffusion operator; the uniform relations discussed in this paper provide alternative sets of nodes to Legendre nodes 
which may be useful in other numerical contexts.

\section{Some preliminary asymptotic results}
\label{prelim}

We briefly recall some asymptotic estimates for the nodes and weights of Gauss-Legendre quadrature. We consider
two types of estimations: elementary asymptotic expansions for large degree, valid in compact subintervals
inside the interval of orthogonality $(-1,1)$, and uniform expansions in terms of Bessel functions. Details on these expansions
are described in \cite{Gil:2019:NIG}. In particular, we use the compound Poincar\'e-type expansions of 
\cite[sect 2.2]{Gil:2019:NIG}, valid in compact subintervals of $(-1,1)$, and the Bessel-type expansion of 
\cite{Frenzen:1985:AUA}, with additional
information on the coefficients given in
\cite[sect 3]{Gil:2019:NIG}. Details on the computation of these expansions, particularly for the weights, are given in 
the Appendices A and B. As an alternative method for computing these asymptotic expansions for the nodes and weights we
mention the reference \cite{Opsomer:2022:AOA} which uses a Riemann-Hilbert approach.  
We only compute the terms we need to prove our results; additional terms  
are computed in \cite{Opsomer:2022:AOA} by using a symbolic computation package and, as expected, our expansions coincide 
with the first terms in \cite{Opsomer:2022:AOA}.
It is not the goal of the 
present paper to rederive asymptotic expansions for the nodes and weights, but to prove new uniform relations between them
which are uniformly valid in $(-1,1)$. 

\subsection{Elementary expansions}

The circle and trapezoid theorems of \cite{Davis:1961:SGT} can be obtained using elementary asymptotic estimates
as $n\rightarrow +\infty$ of the nodes and weights. We recall these results, and give some additional detail
by not only considering the dominant term.

Instead of using as asymptotic parameter $n$, it is more convenient, as done in
\cite{Gil:2019:NIG}, to use $\kappa=n+1/2$. Starting with the nodes, we write $x_i=\cos\theta_i$,
$i=1,\ldots n$, with  
\begin{equation}
\label{expanx}
\theta_i=\alpha_i+\delta_i,\,\alpha_i=\Frac{\kappa+\frac14-i}{\kappa}\pi,\,
\delta_i\sim\Frac{c_1}{\kappa^2}+\Frac{c_2}{\kappa^4}+\cdots.
\end{equation}
Details on the coefficients $c_i$ are given in \cite[sect. 2.2.2]{Gil:2019:NIG}.

For our purpose, we only need the first term in the expansion of $\delta_i$. We have
\begin{equation}
\label{del}
\delta_i=\Frac{\cot\alpha_i}{8\kappa^2}+{\cal O}(\kappa^{-4}),
\end{equation}
which coincides with the expansion in \cite{Gattes:1985:AAE}.

We notice that for the nodes close to the endpoints of the interval of orthogonality
$i$ is small or close to $n$, which means that $c_1={\cal O}(\kappa)$ because
$\cot\alpha_i={\cal O}(\kappa)$; similarly, $c_2={\cal O}(\kappa^3)$, and so on 
(see \cite[Eq. (2.28)]{Gil:2019:NIG}). This indicates
that the expansion for $\delta_i$ is not a valid asymptotic representation for such nodes.

Assuming now that we are not considering extreme nodes, we use (\ref{del}) and get
\begin{equation}
\label{x}
x_i=\cos (\alpha_i+\delta_i)=\cos\alpha_i-\delta_i\sin\alpha_i+{\cal O}(\kappa^{-4})=
F(\kappa) \cos\alpha_i +{\cal O}(\kappa^{-4}),
\end{equation}
where
\begin{equation}
\label{fka}
F(\kappa)=1-\Frac{1}{8\kappa^2}.
\end{equation}

For obtaining the corresponding expansion for the weights we must substitute the expansion for the nodes
 in the
expression for the weights in terms of the derivative of the Legendre polynomial, using the expansion for large order
of the polynomial and re-expanding the result (see Appendix \ref{expanwe}). We get the following expression for the
first terms:
\begin{equation}
\label{w}
w_i=\Frac{\pi}{\kappa}F(\kappa)\sin\alpha_i +{\cal O}(\kappa^{-5}).
\end{equation}

 \subsection{Bessel-type expansions}
 \label{Bessel}

 As an alternative to the previous elementary expansions, we can consider the expansions of
 \cite{Frenzen:1985:AUA} for Jacobi polynomials and their zeros, and use the corresponding
 expansion for the derivative in order to compute the weights \cite{Gil:2019:NIG}. These expansions
 are not so easy to handle as the elementary expansions because they are given in terms of
 Bessel functions but they have the advantage that they are valid for all the Legendre nodes, also
 for those closest to $\pm 1$.

 Particularizing the result for Jacobi polynomials of \cite{Frenzen:1985:AUA} to the
 Legendre case, the nodes are given by $x_i=\cos\theta_i$ with
 $$
 \theta_i=\beta_i+\delta_i,\,\beta_i=\Frac{j_{n-i+1}}{\kappa}
,\,\delta_i=-\Frac{1}{8\kappa^{2}}\Frac{1-\beta_i\cot \beta_i}{\beta_i}+\beta_i^2 {\cal O}(\kappa^{-3}),
 $$
 where $j_k$ is the $k$-th positive zero of the Bessel function $J_0(x)$. This gives
 $$
 x_{n-i} =F(\kappa)\cos\beta_{n-i}+\Frac{1}{8\kappa^2}\Frac{\sin\beta_{n-i}}{\beta_{n-i}}+\sin\beta_{n-i}\beta_{n-i}^2{\cal O}(\kappa^{-3})+{\cal O}(\kappa^{-4}).
 $$
 Now, in the limit $\kappa\rightarrow +\infty$ for finite $i$, we have $\beta_{n-i}=j_{i+1}/\kappa={\cal O}(\kappa^{-1})$ (notice that, as $i\rightarrow +\infty$, $j_{i+1}\sim i\pi$),
 and expanding in powers of $\beta_i$
 \begin{equation}
 \label{dos}
 x_{n-i}=1-\Frac{\beta_{n-i}^2}{2}+{\cal O}(\kappa^{-4})=1-\Frac{j_{i+1}^2}{2\kappa^2}+{\cal O}(\kappa^{-4}).
 \end{equation}

 Of course, this last approximation cannot be accurate for the nodes close to $0$.
 For example, for the smallest nodes $i\sim n/2$, and for such nodes and large $n$,
 $j_{n-i+1}\sim j_{n/2}\sim  n\pi/2$ (see \cite[Eq. 10.21.19]{NIST}),
 and so $\beta_i\sim \pi /2 +{\cal O}(n^{-1})$ (as corresponds to nodes close to zero);
 then the approximation (\ref{dos}) is inaccurate for such nodes. This approximation can only be used
 for $i<<n$.

 For the weights, we must substitute the expansion for the nodes in the expression for the weights in terms
 of the derivative of Legendre polynomials, using the Bessel-type expansion for this derivative, and
 re-expanding again in inverse powers of $\kappa$ (see Appendix \ref{BesselW}). We get

\begin{equation}
\label{dosdos}
w_{n-i}=\Frac{2}{\kappa^2 J_1 (j_{i+1})^2}\left(1-\Frac{1}{12\kappa^2}-
\Frac{j_{i+1}^2}{6\kappa^2}+{\cal O}(\kappa^{-4})\right).
\end{equation}

Similarly as  commented after Eq. (\ref{dos}), this last approximation is accurate  for $i<<n$.

\section{The circle and trapezoid theorems}
\label{sectcir}

Next, we recall the results of \cite{Davis:1961:SGT} for the case of Gauss-Legendre quadrature,
but giving one additional term in the expansions.

\begin{theorem}
\label{circle}
Let $[a,b]\subset (-1,1)$, then
\begin{enumerate}
\item{}
$\Frac{\kappa w_i}{\pi\sqrt{1-x_i^2}}=1-\Frac{1}{8\sin^2 \alpha_i\kappa^2}+{\cal O}(\kappa^{-4})$
for any $i$ such that $x_i\in [a,b]$.
\item{}$\Frac{2(x_{i+1}-x_i)}{w_{i+1}+w_i}=1+\Frac{\pi^2}{12\kappa^2}+{\cal O}(\kappa^{-4})$ for any
$i$ such that $x_i,\,x_{i+1}\in [a,b]$.
\item{}$\Frac{x_{i+1}-x_{i-1}}{2w_i}=1-\Frac{\pi^2}{6\kappa^2}+{\cal O}(\kappa^{-4})$ for any
$i$ such that $x_{i-1},\,x_{i+1}\in [a,b]$.
\end{enumerate}
\end{theorem}

\begin{proof}
These relations follow from (\ref{x}) and (\ref{w}).
\begin{enumerate}
\item{}With $x_i=\cos\theta_i$, this is the same as proving that
$$
\Frac{\pi\sin\theta_i}{\kappa w_i}=1+\Frac{1}{8\sin^2\alpha_i\kappa^2}+{\cal O}(\kappa^{-4}),
$$
which follows from (\ref{w}) and the fact that
$$
\sin\theta_i=\sin(\alpha_i+\delta_i)=\sin(\alpha_i)+\delta_i\cos\alpha_i +{\cal O}(\kappa^{-4}),
$$
which, using (\ref{del}), gives
$$
\sin\theta_i =F(\kappa)\sin\alpha_i+\Frac{1}
{8\sin \alpha_i\kappa^2}+{\cal O}(\kappa^{-4}).
$$
With this and Eq. (\ref{w}) the result is proved.

\item{} The expansion (\ref{x}) and some elementary trigonometry gives
$$
\begin{array}{ll}
x_{i+1}-x_i&=F(\kappa)\left(\cos\alpha_{i+1}-\cos\alpha_i\right)
+{\cal O}(\kappa^{-4})\\
&=2F(\kappa)\sin\left(\Frac{\alpha_i-\alpha_{i+1}}{2}\right)
\sin\left(\Frac{\alpha_i+\alpha_{i+1}}{2}\right)
+{\cal O}(\kappa^{-4})\\
&=2F(\kappa)\sin\left(\Frac{i+1/4}{\kappa}\pi \right)
\sin\left(\Frac{\pi}{2\kappa}\right)
+{\cal O}(\kappa^{-4})
\end{array}
$$
and similarly, using (\ref{w})
$$
\Frac{w_i+w_{i+1}}{2}=\Frac{\pi}{\kappa}F(\kappa)\sin\left(\Frac{i+1/4}{\kappa}\pi \right)
\cos\left(\Frac{\pi}{2\kappa}\right)
+{\cal O}(\kappa^{-5}).
$$

Then
$$
\Frac{2(x_{i+1}-x_i)}{w_{i+1}+w_i}=\Frac{2\kappa}{\pi}\tan\left(\Frac{\pi}{2\kappa}\right)
+{\cal O}(\kappa^{-4}),
$$
and the first two terms in the expansion as $\kappa\rightarrow +\infty$ prove the result.

\item{} First, we have, using (\ref{x})
$$
\begin{array}{ll}
\Frac{x_{i+1}-x_{i-1}}{2}&=F(\kappa)\sin\left(\Frac{\alpha_{i-1}-\alpha_{i+1}}{2}\right)
\sin\left(\Frac{\alpha_{i-1}+\alpha_{i+1}}{2}\right)+{\cal O}(\kappa^{-4})\\
& \\
&=F(\kappa)\sin\left(\Frac{\pi}{\kappa}\right)\sin\alpha_i+{\cal O}(\kappa^{-4}).
\end{array}
$$
Then, because $w_i=\Frac{\pi}{\kappa}F(\kappa)\sin\alpha_i+{\cal O}(\kappa^{-5})$, we get
$$
\Frac{x_{i+1}-x_{i-1}}{2w_i}=\Frac{\kappa}{\pi}\sin\left(\Frac{\pi}{\kappa}\right)+{\cal O}(\kappa^{-4}),
$$
and the first two terms in the expansion of the sine function give the result.

\end{enumerate}
\end{proof}

The first result in the previous theorem is related to the circle theorem, which states that $\Frac{1}{\pi}{nw_i}\sim \sqrt{1-x_i^2}$ as
$n\rightarrow +\infty$, that is,
that
$$
\displaystyle\lim_{n\rightarrow +\infty}\Frac{nw_i}{\pi \sqrt{1-x_i^2}}=1
$$
for any $i$ such that $x_i$ is inside a fixed compact subinterval of $(-1,1)$. We see that the
condition that it holds inside a compact given subinterval is important. Indeed, the extreme
nodes tend to $\pm 1$ as
$n\rightarrow +\infty$, and for those the result does not hold. We observe that in the previous theorem
we have
$$
\Frac{\kappa w_i}{\pi\sqrt{1-x_i^2}}=1-\Frac{1}{8\sin^2 \alpha_i\kappa^2}+{\cal O}(\kappa^{-4})
$$
and for $i=n$ (largest node) the second term in the right hand side does not tend to zero because
$\sin\alpha={\cal O}(\kappa^{-1})$. This indicates that the elementary expansion fails for such zero, and
suggests that the circle theorem does not hold for the largest node;
the same would be true for the second largest zero, and for finite number of zeros close to the endpoints of the
interval of orthogonality. For
the rest of properties in the theorem, the failure is not so evident,
but the next terms in the expansions (not shown) do display such problems. These results are neither valid for the extreme zeros, although the
deviation is smaller in that cases.

To see more explicitly how these results fail, we can use the asymptotics in terms of Bessel functions. For the
 circle theorem, we have
\begin{equation}
\label{lass}
\Frac{\kappa w_{n-i}}{\pi \sqrt{1-x_{n-i}^2}}=
\Frac{\kappa w_{n-i}}{\pi \sin\beta_{n-i}(1+{\cal O}(\kappa^{-2}))}\sim\Frac{2}{\pi j_{i+1}J_{1}(j_{i+1})^2}<1.
\end{equation}
That this quantity is smaller than $1$, therefore violating the circle theorem, is easy 
to prove 
(see Appendix \ref{Sonin}). Defining
\begin{equation}
\label{ai}
a_i=1-\Frac{2}{\pi j_{i+1}J_{1}(j_{i+1})^2},
\end{equation}
one has that this is a positive and decreasing sequence (Appendix \ref{Sonin}), 
and the circle theorem is violated maximally at the largest
node, where $a_0=0.0177079\ldots$,
while for the next node we have $a_1=0.0039098\dots$

With respect to the second result of the theorem, and considering the approximations (\ref{dos}) and (\ref{dosdos}),
we have that
$$
\Frac{w_{n-i}+w_{n-i-1}}{2(x_{n-i}-x_{n-i-1})}\sim\Frac{2}{j_{i+2}^2-j_{i+1}^2}
\left(\Frac{1}{J_1(j_{i+2})^2}+\Frac{1}{J_1(j_{i+1})^2}\right)\equiv b_i +1.
$$
We have checked numerically that the the values $b_i$  
constitute a positive decreasing sequence and the second result of the previous theorem is
violated maximally at the largest zero, where $b_0=0.0002756\ldots$
Finally, for the third result we have
$$
\Frac{x_{n-i+1}-x_{n-i-1}}{2w_{n-i}}\sim \frac18 J_i (j_{i+1})^2 (j_{i+2}^2-j_i^2)\equiv c_i +1
$$
and the values $c_i$  constitute a positive decreasing sequence, as checked numerically. 
The third result of the previous theorem is
violated maximally at the largest zero, where $c_0=0.00010624\ldots$. 

The third result is heuristically 
related to the trapezoidal rule in \cite{Davis:1961:SGT}. Both the second and the third results of
the previous theorem are mentioned as trapezoid theorems in \cite{Davis:1961:SGT}.

\subsection{A first uniform relation: the uniform circle theorem for Gauss-Legendre quadrature}
\label{AFU}

It turns out that it is quite easy to find a uniform version of the circle theorem. We enunciate this
result, although we must recognize that the proof is not complete in the sense we will discuss next.

\begin{conjecture}[Uniform circle theorem]
\label{uct}
The points $(x_i,y_i)\equiv (x_i,\kappa w_i/\pi)$, $i=1,\ldots n$, where $x_i$ and $w_i$ are respectively the nodes and weights of the 
$n$-points Gauss-Legendre quadrature, lie asymptotically on the unit circle. Furthermore
$$
0<1-(x_i^2+y_i^2)<\Frac{1}{4\kappa^2},\,i=1,\ldots n.
$$
\end{conjecture}

The asymptotic evidence supporting this result is strong, as we see next. To begin with, the first part of 
the theorem is trivially 
true for the nodes in any fixed compact subinterval of $(-1,1)$. Considering  (\ref{x}) and (\ref{w}) we have
that
$$
x_i^2+y_i^2=F(\kappa)^2+{\cal O}(\kappa^{-4})=1-\Frac{1}{4\kappa^2}+{\cal O}(\kappa^{-4}),
$$
and therefore the points $(x_i,y_i)$
lie asymptotically on the unit circle.
That the same holds for the nodes close to $+1$ (or $-1$) can be checked by using (\ref{dos}) and (\ref{dosdos}), which
gives
\begin{equation}
\label{newcircl}
x_{n-i}^2+y_{n-i}^2=1-\Frac{k_i}{\kappa^2}+{\cal O}(\kappa^{-4}),\,k_i=j_{i+1}^2 \left(1-\Frac{4}{\pi^2j_{i+1}^2 J_1(j_{i+1})^4}\right),
\end{equation} 

The constant $k_i$ is positive as a consequence of the properties of $a_i$ that are shown
in  Appendix \ref{Sonin}; in addition, it appears that $k_i$ is monotonically increasing as a function of $i$, which we have 
not proven so far, and it is easy to check that
\begin{equation}
\label{limi4}
\displaystyle\lim_{i\rightarrow +\infty}k_i=\Frac{1}{4}
\end{equation}
using the asymptotic
expansion for Bessel functions of large argument \cite[10.17.3]{NIST}, together with the McMahon expansion of the zeros 
$j_{i+1}$ for large $i$ \cite[10.21.19]{NIST}. An even more direct way to check this is by using the direct 
expansion given in \cite[Eq. (4.1)]{Bogaert:2014:IFC} for $J_1 (j_k)^2$ for large $k$ (precisely obtained by combining the aforementioned expansions), together with 
the McMahon expansion for the zeros. Substituting both expansions in the definition of $k_i$ of Eq. (\ref{newcircl}) and reexpanding we get
$$
k_i=\Frac{1}{4} -\Frac{7}{16 \mu_i^2} + {\cal O}(\mu_i^{-4}),
$$
where $\mu_i=\pi (i+3/4)$. This proves the limit (\ref{limi4}) and suggests (but does not prove) that the sequence is increasing.

The (so far unproven) monotonicity of $k_i$ together with (\ref{limi4}) imply the inequalities stated in the previous conjecture.	 
This is a conjecture not only because we have not proved the monotonicity of the 
sequence $\left\{k_i\right\}$
 but also because we have a separate proof for the nodes 
inside any fixed compact subinterval of $(-1,1)$ 
and for the extreme nodes. One would need to make a 
unified analysis for all the nodes in order to arrive to a rigurous proof, using only the Bessel-type expansion and
without the type of re-expansion considered in (\ref{dos}), so that the asymptotics hold for all the nodes and weights. 
Notice that, even with the re-expansion we need to assume some conjectured 
properties of Bessel functions for which there is no proof so far (although we have managed to prove one of them, as 
we will see). 
This complete and fully rigurous 
analysis is appealing and highly challenging, but outside the scope of the present work.

Similar discussions can be considered for other conjectures
we will present later. In all these cases, apart from the strong asymptotic evidence presented, the conjectures have been numerically validated
to high accuracy with an arbitrary precision implementation of the
 algorithm for Gauss-Gegenbauer quadrature of \cite{Gil:2021:FAR}.

  We can re-write the inequalities in Conjecture \ref{uct} as follows
  \begin{corollary}
  The nodes $x_i$ and weights $w_i$ of Gauss-Legendre quadrature satisfy
  $$
  \displaystyle\sqrt{1-\Frac{1}{4\kappa^2 (1-x_i^2)}}<\Frac{y_i}{\sqrt{1-x_i^2}}<1,\,i=1,\ldots n
  $$
  \end{corollary}
  This corollary shows, again,
  that the ``classical" circle theorem does not hold for the extreme nodes because for such nodes
  $1-x_i^2={\cal O}(\kappa^{-2})$.

\section{Uniform relations involving sums of weights}
\label{uniformr}

The results we will next describe were originally motivated by the role they 
 play in certain finite difference schemes used in the discretization of the angular Fokker-Planck
 diffusion operator (see section \ref{nuclear}). As we will see, the proof of these results relies 
 in part in the trapezoid theorems
 and therefore can be interpreted as a uniform counterpart of those results.
 
 We start this section by considering a set of definitions and notations.

Given the Legendre nodes $\{x_i\}_{i=1}^n$, we define the intermediate nodes $\{\bar{x}_i\}_{i=1}^{n-1}$ by
$$
\bar{x}_i=\Frac{x_i+x_{i+1}}{2}.
$$

We observe that
\begin{equation}
\label{prom}
\bar{x}_{i+1}=\bar{x}_i+\Frac{x_{i+2}-x_i}{2}.
\end{equation}

We will call these primary intermediate nodes in order to distinguish them from the secondary intermediate nodes
$\{\bar{z}_i \}$ that we define as follows:
$$
\bar{z}_0=-1,\,\bar{z}_i =\bar{z}_{i-1}+w_{i},\,i=1,\ldots n,
$$
that is,
\begin{equation}
\label{depr}
\bar{z}_{i}=-1+\displaystyle\sum_{j=1}^i w_j=1-\displaystyle\sum_{j=i+1}^n w_j.
\end{equation}
where we have used that $\displaystyle\sum_{i=1}^n w_i=2$ in the last equality. Notice that $\bar{z}_n =1$.

Additionally, we define the secondary nodes
$$
z_i=\Frac{\bar{z}_{i-1}+\bar{z}_i }{2},\,i=1,\ldots n,
$$
that is
\begin{equation}
\label{secon}
z_i=1-\displaystyle\sum_{j=i+1}^n w_j-\frac12 w_i =
z_{i}=-1+\displaystyle\sum_{j=1}^{i-1} w_j+\frac12 w_i.
\end{equation}

We observe that

\begin{equation}
\label{rela}
z_{i+1}=z_i+\Frac{w_i+w_{i+1}}{2}.
\end{equation}

\begin{remark}
\label{obse}
If $n=2k$, $k\in {\mathbb N}$, $\bar{z}_{k}=0$, $\bar{z}_{k+1}=w_{k+1}$ and then
$z_{k+1}=\frac12 w_{k+1}$ is the smallest positive secondary node. On the other hand,
if $n=2k+1$, $k\in {\mathbb N}$, then $z_{k+1}=0$ and the smallest positive secondary node is
$z_{k+2}=\frac12 (w_{k+1}+w_{k+2})$.

Of course, both the intermediate nodes and the secondary nodes are symmetrical with respect to $x=0$.
\end{remark}

Summarizing, we have defined the following sets of nodes:
the {\em (primary) Legendre nodes} $\{x_i\}_{i=1}^n$,
the {\em secondary nodes} $\{z_i\}_{i=1}^n$, the {\em primary intermediate nodes} 
$\{\bar{x}_i\}_{i=1}^{n-1}$ (which are
trivially interlaced with the primary nodes) and the {\em secondary intermediate nodes} 
$\{\bar{z}_i\}_{i=0}^n$, which
we will prove that are also interlaced with the primary nodes (and are trivially interlaced with the secondary
nodes). We will prove that the secondary nodes $\{z_i\}_{i=1}^n$ approach asymptotically the primary nodes,
with a very uniform error term which is ${\cal O}(\kappa^{-2})$, and that the same will be true with respect to
the primary and secondary intermediate nodes  $\{\bar{x}_i\}_{i=1}^{n-1}$ and $\{\bar{z}_i\}_{i=0}^n$.

\subsection{Relation between the primary and the secondary nodes}

That the secondary nodes approximate the primary nodes $x_i$ as
$n\rightarrow +\infty$ is one of the three uniform relations we will discuss in this section. We
start by proving this fact for the smallest positive nodes (we don't need to consider negative nodes for obvious reasons).

\begin{lemma}
Let $x_{i}$ be the smallest positive node, then
$$
\Frac{x_i}{z_i}=1+\Frac{\pi^2}{12\kappa^2}+{\cal O}(\kappa^{-4}).
$$
\end{lemma}

\begin{proof}
If $n=2k$, $x_{k+1}$ is the smallest positive node and $x_{k}=-x_{k+1}$, $w_{k+1}=w_k$. In addition, as discussed in
Remark \ref{obse}, $z_{k+1}=\frac12 w_{k+1}$. Then
$$
\Frac{x_{k+1}}{z_{k+1}}=\Frac{2x_{k+1}}{w_{k+1}}=\Frac{x_{k+1}-x_k}{w_{k+1}}=2\Frac{x_{k+1}-x_k}{w_{k+1}+w_k},
$$
and considering the second result of Theorem \ref{circle} completes the proof for even $n$.

For $n=2k+1$, as discussed in Remark \ref{obse}, $x_{k+1}=z_{k+1}=0$, and the smallest positive node is
$x_{k+2}$. Now, using (\ref{rela}),
$$
z_{k+2}=z_{k+1}+\Frac{w_{k+2}+w_{k+1}}{2}=\Frac{w_{k+2}+w_{k+1}}{2},
$$
then
$$
\Frac{x_{k+2}}{z_{k+2}}=2\Frac{x_{k+2}-x_{k+1}}{w_{k+2}+w_{k+1}}
$$
and again the second result of Theorem \ref{circle} completes the proof.

\end{proof}

\begin{lemma}
\label{previo}
If, as $\kappa\rightarrow +\infty$, $$\Frac{x_j}{z_j}=1+\Frac{\pi^2}{12\kappa^2}+{\cal O}(\kappa^{-4})$$
for $j=m$ then the
same holds for $j=m\pm 1$, provided all the nodes are inside a fixed compact subset of $(-1,1)$.
\end{lemma}
\begin{proof}
We assume that $\Frac{x_m}{z_m}=1+\Frac{\pi^2}{12\kappa^2}+{\cal O}(\kappa^{-4})$, consider the
relation (\ref{rela}) and then use the
 second result of Theorem \ref{circle}:
$$
\begin{array}{ll}
z_{m+1}&=z_m+\Frac{w_m+w_{m+1}}{2}\\
&=x_m
\left(1-\Frac{\pi^2}{12\kappa^2}+{\cal O}(\kappa^{-4})\right)+(x_{m+1}-x_m)\left(1-\Frac{\pi^2}{12\kappa^2}+{\cal O}(\kappa^{-4})\right)\\
&=x_{m+1}
\left(1-\Frac{\pi^2}{12\kappa^2}+{\cal O}(\kappa^{-4})\right).
\end{array}
$$
This proves the result for $j=m+1$. That the same holds for $j=m-1$ is obvious.
\end{proof}

\begin{remark}
\label{oddc}
If $n$ is odd $x_{j}=z_{j}=0$ for $j=(n+1)/2$ and for this reason one should exclude the trivial case 
$x_j=0$ in the previous lemma. The same
will be true for the rest of results involving ratios: when the denominator is zero, the numerator is also zero. Instead of singling out this
case for each of these results, we will assume that such case is excluded.
\end{remark}

\begin{corollary}
\label{coro}
Let $[a,b]\subset (-1,1)$ and $x_k,\,z_k \in [a,b]$, then
$$
\Frac{x_k}{z_k}=1+\Frac{\pi^2}{12\kappa^2}+{\cal O}(\kappa^{-3})
$$
as $\kappa \rightarrow +\infty$.
\end{corollary}

The error term in the previous Corollary
 is $n$ times ${\cal O}(\kappa^{-4})$, which gives ${\cal O}(\kappa^{-3})$.
 However, it appears that it is in fact ${\cal O}(\kappa^{-4})$, 
 as the Bessel-type expansions will indicate.

 Corollary \ref{coro}, similarly to Theorem \ref{circle}, holds
for nodes in a compact interval inside $(-1,1)$. However, we can check that the quantities $x_i/z_i-1$
are in fact ${\cal O}(\kappa^{-2})$ for all the nodes in $(-1,1)$ and, as the next results suggest, with an asymptotic constant not larger than
$\pi^2 /12$. For this purpose, we use the Bessel-type expansions of section \ref{Bessel}.

We consider the first equality in (\ref{secon}) together with the expansions (\ref{dos}) and (\ref{dosdos}). We have
$$
\Frac{x_{n-i}}{z_{n-i}}-1=\Frac{1}{\kappa^2}C_i+{\cal O}(\kappa^{-4}),
$$
with
\begin{equation}
C_i=2\displaystyle\sum_{k=0}^{i-1}J_1 (j_{k+1})^{-2}+J_1 (j_{i+1})^{-2}-\Frac{j_{i+1}^2}{2},
\end{equation}
where the sum is assumed to be zero for $i=0$.

For the largest zero we have $C_0=0.8187877\ldots$,
which is quite close to the error constant for the smallest nodes,
which is $C=\pi^2/12=0.8224670\ldots$ In fact, the following conjecture appears to be true, as numerical experiments
show.
\begin{conjecture}
\label{conjc}
The sequence $\{C_i\}_{i=0}^{+\infty}$ is positive, increasing and with limit $\displaystyle\lim_{i\rightarrow \infty}
C_i=\Frac{\pi^2}{12}$.
\end{conjecture}
In the Appendix \ref{apc} we prove that the limit $C_{+\infty}$ is finite.

Conjecture \ref{conjc} leads us to the following conjecture.
\begin{conjecture}
\label{mainc1}
The primary and secondary nodes of Gauss-Legendre quadrature of degree $n$ are such that
$$
0<\Frac{x_i}{z_i}-1<\Frac{\pi^2}{12\kappa^2},\,i=1,\ldots n,
$$
with $\pi^2/12$ the best possible asymptotic constant.
\end{conjecture}

Conjecture \ref{mainc1}
can be used to check that the primary nodes and the secondary intermediate nodes are
interlaced, that is $\bar{z}_0<x_1<\bar{z}_1<x_2<\cdots <\bar{z}_{n-1}<x_n<
\bar{z}_n$.

\begin{conjecture}
\label{interlace-result}
The primary nodes $\{x_i\}$ and the secondary
intermediate nodes $\{\bar{z}_i\}$ are interlaced, at least asymptotically.
\end{conjecture}

\begin{proof}
We are proving the interlacing asymptotically, that is, as $\kappa \rightarrow +\infty$, and
assuming Conjecture \ref{mainc1}. However the result
appears to hold for any $n$.

We are checking that in each interval $(\bar{z}_{j-1},\bar{z}_{j})$, $j=1,\ldots n-1$
there is one primary node $x_{j}$; because there
are $n+1$ intermediate nodes (including $\pm 1$) and
$n$ primary nodes, this proves interlacing. For proving this, and because $\bar{z}_{j}-\bar{z}_{j-1}=w_j$, we need to prove that $|x_j-z_j|/(w_j/2)$ is smaller than $1$, where
$z_j=(\bar{z}_{j-1}+\bar{z}_{j})/2$ is the middle point of the
interval  $(\bar{z}_{j-1},\bar{z}_{j})$.

We use that $x_j/z_j-1<\Frac{\pi^2}{12\kappa^2}$ (conjecture \ref{mainc1}) and then
\begin{equation}
\label{bodis}
|x_j-z_j|<|z_j|\Frac{\pi^2}{12\kappa^2}<\Frac{\pi^2}{12\kappa^2},
\end{equation}
and considering the elementary asymptotic for $w_j$ we have
$$
\Frac{2|x_j-z_j|}{w_j}<\Frac{\pi^2}{6\kappa^2 w_j}\sim \Frac{\pi}{6\kappa \sin\alpha_j},
$$
which is smaller than $1$ for large enough $\kappa$ and fixed $j$.
This would prove the result under the same conditions as in Theorem \ref{circle}, but
not for the extreme nodes.

For proving that this holds in all the
interval $(-1,1)$ we have to rely on the Bessel expansions for the extreme nodes 
again. The distance
$|x_j-z_j|$ can be uniformly bounded
in all the interval using (\ref{bodis}); however, the length of the interval
$(\bar{z}_{j-1},\bar{z}_{j})$ 
decreases because the positive weights $w_j$ decrease with increasing $j$, and we expect that the most problematic
case is for $j=n$.

Using Bessel asymptotics for $w_{n-i}$ (Eq. (\ref{w}))
$$
\Frac{2|x_{n-i}-z_{n-i}|}{w_{n-i}}\lesssim \Frac{\pi^2 J_1(j_{i+1})^2}{12}\equiv D_i.
$$

For the last interval ($i=0$) we have $D_0=0.2216664\ldots$, for the following $D_1=0.0952253\ldots$ and
$D_i$ tends to zero as $i$ increases. 

Considering the result proved in Appendix \ref{Sonin} we have
$$
D_i< \Frac{\pi}{6j_{i+1}}={\cal O}(i^{-1}).
$$
\end{proof}

\subsection{Relation between the primary and secondary intermediate nodes}

The relation between the intermediate nodes can be shown in a similar way as in the previous section, but in this case
we will use the third result in Theorem \ref{circle} instead of the second. We first prove the following

\begin{lemma}
Let $\{\bar{x}_i\}$ be the smallest positive intermediate node, then
$$
\Frac{\bar{x}_i}{\bar{z}_i}=1-\Frac{\pi^2}{6\kappa^2}+{\cal O}(\kappa^{-4}).
$$
\end{lemma}

\begin{proof}
Let $n=2k$, $k\in {\mathbb N}$; by symmetry we have that
$\bar{x}_k=\bar{z}_k=0$ and then, using (\ref{prom}),
 the smallest positive intermediate node $\bar{x}_{k+1}$ is such that
$$
\bar{x}_{k+1}=\Frac{x_{k+2}-x_k}{2},
$$
and by the third result of Theorem \ref{circle}
$$
\bar{x}_{k+1}=w_{k+1}\left(1-\Frac{\pi^2}{6\kappa^2}+{\cal O}(\kappa^{-4})\right).
$$
Because $\bar{z}_{k+1}=\bar{z}_k + w_{k+1}=w_{k+1}$ the result is proved for $n$ even.

Take now $n=2k+1$, $k\in {\mathbb N}$.
Then $x_k=-x_{k+2}$ and $x_{k+1}=0$ and by the third result of
 Theorem \ref{circle}
 $$
 \Frac{x_{k+2}-x_k}{2w_{k+1}}=\Frac{x_{k+2}}{w_{k+1}}=\Frac{2\bar{x}_{k+1}}{w_{k+1}}.
 $$
 Now, by symmetry $\bar{z}_{k}=-\bar{z}_{k+1}$ and because $\bar{z}_{k+1}=\bar{z}_k + w_{k+1}$,
 $\bar{z}_{k+1}=w_{k+1}/2$. Therefore
 $$
 \Frac{x_{k+2}-x_k}{2w_{k+1}}=\Frac{\bar{x}_{k+1}}{\bar{z}_{k+1}},
 $$
 and applying the third result of Theorem \ref{circle} to the left-hand side
 we have the desired result.
\end{proof}

\begin{lemma}
\label{previo2}
If, as $\kappa\rightarrow +\infty$, $$\Frac{\bar{x}_j}{\bar{z}_j}=1-\Frac{\pi^2}{6\kappa^2}+{\cal O}(\kappa^{-4})$$
for $j=m$ then the
same holds for $j=m\pm 1$, provided all the nodes are inside a fixed compact subset of $(-1,1)$.
\end{lemma}

\begin{proof}

Considering the property (\ref{prom}) we have
$$
\bar{x}_{m+1}=\bar{x}_m+\Frac{x_{m+2}-x_m}{2}.
$$
Now, because we are assuming that
$\Frac{\bar{x}_m}{\bar{z}_m}=1-\Frac{\pi^2}{6\kappa^2}+{\cal O}(\kappa^{-4})$, using the third result of Theorem \ref{circle} we have
$$
\bar{x}_{m+1}=\bar{z}_m
\left(1-\Frac{\pi^2}{6\kappa^2}+{\cal O}(\kappa^{-4})\right)
+w_{m+1}\left(1-\Frac{\pi^2}{6\kappa^2}+{\cal O}(\kappa^{-4})\right)
$$
and because $\bar{z}_{m+1}=\bar{z}_m+w_{m+1}$ the result is proved for $j=m+1$.
\end{proof}

As a consequence, and similarly as we proved for the primary and secondary nodes in the
previous section, we have

\begin{corollary}
\label{coro2}
Let $[a,b]\subset (-1,1)$ and $x_k,\,z_k \in [a,b]$, then
$$
\Frac{\bar{x}_k}{\bar{z}_k}=1-\Frac{\pi^2}{6\kappa^2}+{\cal O}(\kappa^{-3})
$$
as $k\rightarrow +\infty$.
\end{corollary}

And the same comments following Corollary \ref{coro} apply here.

Using the Bessel-type expansions we can confirm, similarly as in the previous section,
that the relation between primary and secondary intermediate nodes holds uniformly in
all the interval $(-1,1)$, and not only in fixed compact subsets.

Considering the second equality in (\ref{depr}) together with the expansions (\ref{dos}) and (\ref{dosdos}) we have
$$
\Frac{\bar{x}_{n-i}}{\bar{z}_{n-i}}-1=\Frac{1}{\kappa^2}E_i+{\cal O}(\kappa^{-4}), i=1,2,...
$$
with
\begin{equation}
E_i=2\displaystyle\sum_{k=1}^{i}J_1 (j_{k})^{-2}-\Frac{j_{i}^2}{4}-\Frac{j_{i+1}^2}{4}.
\end{equation}

For the largest intermediate node ($i=1$) we have $D_1=-1.6428507\ldots$,
which is quite close to the error constant for the smallest nodes,
which is $E_i=-\pi^2/12=-1.6449340\ldots$. In fact, the following conjecture appears to be true, as numerical experiments
show.
\begin{conjecture}
The sequence $\{E_i\}_{i=1}^{+\infty}$ is negative, decreasing and with limit $\displaystyle\lim_{i\rightarrow +\infty}
E_i=-\Frac{\pi^2}{6}$.
\end{conjecture}
That the limit is finite can be proved similarly as it is proved in
appendix \ref{apc} that the limit $C_{+\infty}$ is finite. We omit the proof for brevity.

The previous result leads to the following conjecture.
\begin{conjecture}
\label{mainc1new}
The primary and secondary intermediate
nodes of Gauss-Legendre quadrature of degree $n$ are such that
$$
0<1-\Frac{\bar{x}_i}{\bar{z}_i}<\Frac{\pi^2}{6\kappa^2},\,i=1,\ldots n-1,
$$
with $\pi^2/6$ the best possible asymptotic constant.
\end{conjecture}

\subsection{Relation between the first order partial moments and the intermediate nodes}

Let us now define $\alpha_0=0$, $\alpha_{i}=\alpha_{i-1} -2x_i w_i$, that is,
$$
\alpha_{n-i}=2\displaystyle\sum_{k=0}^{i-1} x_{n-k} w_{n-k}=2\displaystyle\sum_{k=i +1}^n x_k w_k.
$$
Notice that the previous two sums are equal because $x_{n-k+1}=-x_k$ and $w_{n-k+1}=w_{k}$, $k=1,\ldots n$, which also 
implies that the first order moment is zero, that is $\displaystyle\sum_{k=1}^n x_k w_k=0$. For the same reason
$\alpha_i=\alpha_{n-i}$, $i=0,\ldots n$.

We propose the following conjecture relating the first order partial moments with the secondary intermediate nodes.

\begin{conjecture}
\label{lili4}
The following holds for all the secondary intermediate nodes and partial moments:
$$
0<\Frac{\alpha_i}{1-(\bar{z}_i)^2}-1<\Frac{\pi^2}{12 \kappa^2},
$$
where the constant $\Frac{\pi^2}{12}$ is the best possible.
\end{conjecture}

The asymptotic evidence supporting this conjecture is not as solid as for the case of conjecture \ref{mainc1}, as
it will be solely based on Bessel asymptotics for large nodes, and assuming another unproved
property for Bessel functions. However, later we discuss how this can be also seen as a consequence of
the asymptotic relation between primary and secondary nodes; for this purpose we solve an initial value
problem relating both results.

Now we see how the Bessel-type expansions suggest the validity of Conjecture \ref{lili4}.
Taking into account (\ref{dos}) and (\ref{dosdos}) we have
$$
x_{n-k}w_{n-k}=\Frac{2}{\kappa^2 J_1 (j_{k+1})^2}\left(1-\Frac{1}{12\kappa^2}-\Frac{2j_{k+1}^2}{3\kappa^2}\right)+{\cal O}(\kappa^{-5}),
$$
and therefore
$$
\alpha_{n-i}=\Frac{4G(\kappa)S_i^{(1)}}{\kappa^2}-\Frac{8S_i^{(2)}}{3\kappa^4}+{\cal O}(\kappa^{-5})
$$
where $G(\kappa)=1-\Frac{1}{12\kappa^2}$  and
$$
S^{(1)}_i=\displaystyle\sum_{k=0}^{i-1}\Frac{1}{J_1(j_{k+1})^2},\,S^{(2)}_i=\displaystyle\sum_{k=0}^{i-1}\Frac{j_{k+1}^2}{J_1(j_{k+1})^2}.
$$
On the other hand, considering again (\ref{dosdos})
$$
\bar{z}_{n-i}=1-\displaystyle\sum_{k=0}^{i-1}w_{n-k}=1-\left(\Frac{2G(\kappa)S_i^{(1)}}{\kappa^2}-
\Frac{S_i^{(2)}}{3\kappa^4}\right)+{\cal O}(\kappa^{-5}).
$$
With this two expansions we get
$$
\Frac{\alpha_{n-i}}{1-(\bar{z}_{n-i})^2}
=1+\left(S_i^{(1)}-\Frac{S_i^{(2)}}{2S_i^{(1)}}\right)\kappa^{-2}+{\cal O}(\kappa^{-4})
$$
which proves that
$\Frac{\alpha_{n-i}}{1-(\bar{z}_{n-i})^2}-1={\cal O}(\kappa^{-2})$, at least for the large zeros
(such that $j_i<<\kappa$ holds).

It turns out that the constants
\begin{equation}
K_i=S_i^{(1)}-\Frac{S_i^{(2)}}{2S_i^{(1)}}
\end{equation}
are bounded for any $i\in{\mathbb N}$, as numerical experiments show. We propose the following conjecture.
\begin{conjecture}
$\{K_i\}_{i=1}^{\infty}$ is a monotonically increasing positive sequence and
$$
\displaystyle\lim_{i\rightarrow +\infty}K_i=\Frac{\pi^2}{12}
$$
\end{conjecture}
We observe that $K_1=C_0=0.8187877\ldots$, as could expected because
$$
\Frac{\alpha_{n-1}}{1-(\bar{z}_{n-1})^2}=\Frac{2x_n w_n}{1-(1-w_n)^2}=\Frac{x_n}{1-w_n/2}=\Frac{x_n}{z_n}.
$$
On the other hand, because $K_{+\infty}=\pi^2/12=0.8224\ldots$, we observe that the sequence
$\{K_i\}$, as happened with $\{C_i\}$, has very small variation. The convergence of the
$\{K_i\}$ sequence appears to be fast, and $K_{100}$ already approximates $\pi^2 /12$ with
 $6$ exact digits.

 A consequence of Conjecture \ref{lili4} is that, considering $n$ even, and because by symmetry
$\bar{z}_{n/2}=0$, we
 have that
 $$
 \alpha_{n/2}=2\displaystyle\sum_{x_i>0}x_i w_i =1+\Frac{\pi^2}{12\kappa^2}+{\cal O}(\kappa^{-4}).
 $$
 On other hand, for $n$ odd we have $0=x_k=z_{k}=\bar{z}^{(w)}_{k}-\frac12 w_{k}$ for
 $k=(n+1)/2$ and then
 $$
 \begin{array}{ll}
 \alpha_{k}&=2\displaystyle\sum_{x_i>0}x_i w_i=\left(1-\Frac{w_{k}^2}{4}\right)
 \left(1+\Frac{\pi^2}{12\kappa^2}+{\cal O}(\kappa^{-4})\right),\\
 &= 1-\Frac{\pi^2}{6\kappa^2}+{\cal O}(\kappa^{-4}).
 \end{array}
 $$

 These last results, as well as the rest of asymptotic relations between nodes and weights, can be
 numerically checked to high accuracy with an arbitrary precision implementation of the
 algorithm for Gauss-Gegenbauer quadrature of \cite{Gil:2021:FAR}.

\subsubsection{A complementary observation that supports Conjecture \ref{lili4}}
Notice that $D(x) = 1 - x^2$ is the solution of the following initial value problem (IVP):
\begin{equation}
\label{IVP}
\left\{
\begin{array}{l}
y' = -2x\ \mbox{in}\ (-1,1),\\
y(-1) = 0.
\end{array}
\right.
\end{equation}

\par Our goal is to prove the following observation, which demonstrates that proving that $\max_{0 \leq i \leq n} |D(\bar{z}_i) - \alpha_i| = {\cal O}(n^{-2})$ (cf. Conjecture \ref{lili4}) is as difficult as proving Conjecture \ref{mainc1}, but not more.
\begin{obs}\label{goal} Under Conjecture \ref{mainc1},
\beq
\max_{0 \leq i \leq n} |D(\bar{z}_i) - \alpha_i| = {\cal O}(n^{-2}).
\eeq
\end{obs}

\par To this aim, we are going to solve the IVP (\ref{IVP}) on the mesh of intermediate nodes $\{\bar{z}_i\}_{i = 0}^n$ by means of a convenient numerical method of order 2 which will compute the first order partial moment $\alpha_i$ as the approximation of $D(\bar{z}_i)$. Let us start trying the mid\=/point rule:
\begin{equation}
\label{mid-point-1}
\left\{
\begin{array}{l}
y_0 = 0,\\
y_{i+1} = y_i - 2 h_i (\bar{z}_i + h_i/2)\ \mbox{for}\ i = 0,\dots,n-1,
\end{array}
\right.
\end{equation}
where $h_i = w_{i+1}$ is the distance between $\bar{z}_i$ and $\bar{z}_{i+1}$.

\par Since $\bar{z}_i + h_i/2 = z_{i+1}$, the scheme (\ref{mid-point-1}) does not use the desired formula $y_{i+1} = y_i - 2 w_{i+1} x_{i+1}$, but rather $y_{i+1} = y_i - 2 w_{i+1} z_{i+1}$. This problem can be easily fixed by considering instead the following adjustment:
\begin{equation}
\label{mid-point-2}
\left\{
\begin{array}{l}
y_0 = 0,\\
y_{i+1} = y_i - 2 h_i (\bar{z}_i + h_i^*/2)\ \mbox{for}\ i = 0,\dots,n-1,
\end{array}
\right.
\end{equation}
where $h_i^* = 2 (x_{i+1} - \bar{z}_i)$. This scheme computes $y_i = \alpha_i$ for $i = 0,\dots,n$, whereupon the reasoning is complete and proves Observation \ref{goal}. The statement in Conjecture \ref{mainc1} is needed to guarantee that the step\=/sizes $h_i^*$ and $h_i$ are close enough so that the scheme (\ref{mid-point-2}) still has order 2. A detailed proof is given in the Appendix \ref{ApE}.

\section{Applications in finite difference schemes}

\label{nuclear}

The initial motivation for the study of these properties was the analysis of convergence of a number of finite differece schemes
used in the discretization of the angular Fokker\=/Planck diffusion operator, of particular relevance 
in the context of nuclear engineering. 
The 
properties discussed in this paper appear as necessary conditions for the convergence with order $2$ of some of these methods 
\cite{Lopez:2023:AOD}.
Although these properties are of intrinsic interest, next we briefly explain why they are important for the analysis of such finite difference schemes. These properties were previously either assumed correct or simply overlooked, and the present paper fills this gap by providing strong asymptotic evidence that they hold.

Let us define $D(x) = 1 - x^2$ and suppose that we want to discretize the operator
\begin{equation}
\label{FP-Lap}
\Delta_{\rm{FP}} f(x) = \left(D(x) f'(x)\right)',\quad \mbox{with}\ x \in [-1,1],
\end{equation}
by using, for each natural $n$, the mesh formed by the Gauss\=/Legendre nodes $x_i$. This is a natural choice which 
acquires special meaning if one must solve a PDE of which the above operator is a part and there are quantities of interest defined by means of integrals of functions involving the solution. To accomplish the discretization, Haldy and Ligou \cite{HA-LI} 
used as support the secondary intermediate nodes $\bar{z}_i$, $i = 0,\dots,n$, defined in (\ref{depr}), and, going through the intermediate step
\begin{equation}
\label{HL-previa}
\Delta_{\rm{FP}} f(x_i) \approx \frac{D(\bar{z}_i) f'(\bar{z}_i) - D(\bar{z}_{i-1}) f'(\bar{z}_{i-1})}{w_i},
\end{equation}
they finally employed
\begin{equation}
\label{HL}
\Delta_{\rm{FP}} f(x_i) \approx \frac{D(\bar{z}_i) \frac{f(x_{i+1}) - f(x_i)}{x_{i+1} - x_i} - D(\bar{z}_{i-1}) \frac{f(x_i) - f(x_{i-1})}{x_i - x_{i-1}}}{w_i}
\end{equation}
for $i = 1,\dots,n$, where terms containing the undefined nodes $x_0$ and $x_{n+1}$ are multiplied by zero and must be ignored.

\par Since $w_i = \bar{z}_i - \bar{z}_{i-1}$, the idea behind (\ref{HL-previa}) and (\ref{HL}) is clear as soon as one thinks of the centered formula of two points for the first derivative. It is then apparent that the formula (\ref{HL}) will perform better if
\begin{enumerate}
\item $x_i$ is close to $z_i$, the mid\=/point of $[\bar{z}_{i-1},\bar{z}_i]$ (see Conjecture \ref{mainc1}), and
\item $\bar{z}_i$ is close to $\bar{x}_i$, the mid\=/point of $[x_i,x_{i+1}]$ (see Conjecture \ref{mainc1new}).
\end{enumerate}

\par In particular, the interlacing property stated in Corollary \ref{interlace-result} is needed.

\par On the other hand, we have Morel's scheme \cite{MO}, which is obtained by substituting in (\ref{HL}) $D(\bar{z}_{i-1})$ and $D(\bar{z}_i)$ by the first order partial moments $\alpha_{i-1}$ and $\alpha_i$, respectively. It is natural to think that $\alpha_i$ must be close to $D(\bar{z}_i)$ (see Conjecture \ref{lili4}) so as not to lose accuracy.

\par Experimentally, both schemes converge with order 2. The proof of this fact, which uses the properties just mentioned, can be found in \cite{Lopez:2023:AOD}.

In these methods, a motivation for using nodes defined as sum of weights (what we call secondary nodes and secondary intermediate nodes) was to preserve the zeroth and first moment properties, that is,
\begin{eqnarray}
\label{0thMP} \int_{-1}^1 \Delta_{\rm{FP}} f(x)\ dx = 0,\\
\label{1stMP} \int_{-1}^1 x \Delta_{\rm{FP}} f(x)\ dx = -2 \int_{-1}^1 x f(x)\ dx,
\end{eqnarray}
in the discrete setting. Haldy-Ligou's scheme preserves the zeroth moment, while Morel's scheme preserves both.

We refer to \cite{Lopez:2023:AOD} for further details 
of the analysis of those finite differences schemes, for which the use of the uniform properties described in this paper is crucial. 
The reader might also like to consult references \cite{OL-FR}, \cite{PA-WA-PR}, and \cite{WA-PR} to see that Morel's 
scheme is still in use.

Although we are not aware of additional specific numerical methods, apart from \cite{HA-LI} and \cite{MO},
 where the properties discussed in this paper are used 
(probably because they were unknown), the secondary nodes can be useful in other numerical applications because they are 
closely related to 
Legendre nodes, as we have proved. We observe that in finite difference methods based on Legendre nodes, 
the differences between the nodes are prone to severe rounding errors,
 particularly significant at the extremes of $(-1,1)$, where the nodes tend to cluster. As pointed out by
 Laurie \cite[page 215]{Laurie:2001:COG}: ``For careful work, one should store as floating-point numbers not the nodes, 
 but the gaps between them".
Alternatively, the differences between secondary nodes are sums of positive weights, which can be accurately 
computed either by iterative methods \cite{Gil:2021:FAR} or by asymptotic methods \cite{Gil:2021:FAR} 
(even for a very large number of nodes); 
no cancellations occur in this case.

For a better understanding of the uses of the secondary nodes in approximation, 
it  will be interesting to analyze the Lebesgue constants corresponding to those nodes, both numerically and
 asymptotically. Some preliminary numerical
tests indicate that the Lebesgue functions are very similar for the primary and secondary nodes and that their 
asymptotic behavior is essentially the same. This will be the object of further study.

 \appendix

 \section{Appendix: Elementary expansion for the weights}

 \label{expanwe}

For obtaining the first term of the elementary asymptotic expasion for the weights, we start from Eqs. (4.3), (4.4), (4.5) and (2.20) of
\cite{Gil:2019:NIG} to write
\begin{equation}
\label{expanw}
w_i=\Frac{\pi}{\kappa}\sin\theta_i H_\kappa(\theta_i)^{-2},
\end{equation}
where
$$
H_\kappa(\theta)=\left(1+\Frac{m_2 (\theta)}{\kappa^2}+{\cal O}(\kappa^{-4})\right)\sin\chi+
\left(\Frac{n_1(\theta)}{\kappa}+{\cal O}(\kappa^{-3})\right)\cos\chi,
$$
with $\chi=\kappa\theta-\Frac{\pi}{4}$ and
$$
m_2(\theta)=\Frac{1}{384}\Frac{24-3\cos^2\theta}{\sin^2\theta},\,n_1(\theta)=-\Frac{1}{8}\cot\theta.
$$
Now, for obtaining an expansion for the weights we must substitute the expansion (\ref{expanx}) in
(\ref{expanw}) and re-expand in powers of $\kappa^{-2}$. For our purpose, it will be enough to use the
terms explicitly shown.

In the expression for $m_2$ and $n_1$ it will be enough to replace $\theta_i$ by $\alpha_i$. Now, we have
$$
\chi_i=\kappa\theta_i-\Frac{\pi}{4}=\left(n-i+\frac12\right)\pi+d_i,\,d_i=\Frac{\cot\alpha_i}{8\kappa},
$$
and with this
$$
\sin\chi_i=(-1)^{n-i}\left(1-\Frac{d_i^2}{2}+{\cal O}(\kappa^{-4})\right),\,
\cos\chi_i=-(-1)^{n-i}d_i+{\cal O}(\kappa^{-3})
$$
Substituting the expansions
$$
\begin{array}{ll}
(-1)^{n-i}H_\kappa (\theta_i)&=1+\Frac{m_2(\alpha_i)}{\kappa^2}
-\Frac{d_i^2}{2}
-\Frac{d_i n_1(\alpha_i)}{\kappa}+{\cal O}(\kappa^{-4})\\
&= 1+\Frac{1}{16\kappa^2\sin^2\alpha_i}+{\cal O}(\kappa^{-4})
\end{array}
$$
Therefore
\begin{equation}
\label{f1}
H_\kappa(\theta_i)^{-2}=1-\Frac{1}{8\kappa^2\sin^2\alpha_i}+{\cal O}(\kappa^{-4}).
\end{equation}
Finally, we have
\begin{equation}
\label{f2}
\begin{array}{ll}
\sin\theta_i&=\sin(\alpha_i+\delta_i)=\sin(\alpha_i)+\delta_i\cos(\alpha_i)+{\cal O}(\kappa^{-4})\\
&=\sin\alpha_i\left(1+\Frac{1}{8\kappa^2\sin^2\alpha_i}-\Frac{1}{8\kappa^2}\right)
\end{array}
\end{equation}
and substituting (\ref{f1}) and (\ref{f2}) in (\ref{expanw}) we get
\begin{equation}
w_i=\Frac{\pi}{\kappa}F(\kappa)\sin\alpha_i +{\cal O}(\kappa^{-5}),
\end{equation}
with $F(\kappa)$ defined in (\ref{fka}). This coincides with the expansion
\cite[A.8]{Opsomer:2022:AOA}, which is obtained by different means (using the Riemann-Hilbert approach).

\section{Appendix: Bessel-type expansion for the weights}

\label{BesselW}

For obtaining the Bessel-type expansions for the Gauss-Legendre weights we start by
 combining eqs. (4.2), (4.4), (4.6) and (3.10) of \cite{Gil:2019:NIG}, to write
 \begin{equation}
 \label{weBe}
 w_i=\Frac{2\sin\theta_i}{\kappa^2 \theta_i} H_{\kappa}(\theta_i)^{-2},
 \end{equation}
where
$$
H_{\kappa}(\theta)=J_1 (\kappa \theta)Y(\theta)-\Frac{1}{2\theta\kappa}
J_0 (\kappa\theta)Z(\theta)
$$
with $Y(\theta)$ and $Z(\theta)$ admitting asymptotic expansions in powers of $\kappa^{-2}$
$$
Y(\theta)\sim \displaystyle\sum_{m=0}^{+\infty}\Frac{Y_m (\theta)}{\kappa^{2m}},\,
Z(\theta)\sim \displaystyle\sum_{m=0}^{+\infty}\Frac{Z_m (\theta)}{\kappa^{2m}}.
$$
 Because $\theta_i=\beta_i+\delta_i$
with $\delta_i={\cal O}(\kappa^{-2})$, we have $Y(\theta_i)=1+Y_1(\beta_i)\kappa^{-2}+
{\cal O}(\kappa^{-4})$ and $Z(\theta_i)=1+{\cal O}(\kappa^{-2})$. For our purposes,
we will only need to known that $Y_0 (\theta)=Z_0 (\theta)=1$ and that
$Y_1(\beta_i)=\frac{1}{48}+{\cal O}(\beta_i^2)$ as $\beta_i\rightarrow 0$.

On the other hand $\kappa \theta_i=b_i+d_i$, with $b_i=j_{n-i+1}$ and then $J_0(b_i)=0$; in addition, using
the differentiation formulas \cite[10.6.2]{NIST} we have $J_1' (b_i)=-J_1(b_i)/b_i$,
$J_1''(b_i)=(2/b_i^{2}-1)J_1(b_i)$,  $J_0' (b_i)=-J_1(b_i)$ and then
$$
\begin{array}{l}
J_1(\kappa \theta_i)=J_1(b_i)\left(1-\Frac{d_i}{b_i}+\Frac{d_i^2}{b_i^2}-\frac12 d_i^2+{\cal O}(d_i^3)\right),\\
J_0(\kappa \theta_i)=J_1 (b_i)\left(-d_i+\Frac{d_i^2}{2b_i}+{\cal O}(d_i^3)\right).
\end{array}
$$

Putting this together and using that $Y_0(\theta)=Z_0 (\theta)=1$ gives
\begin{equation}
\label{HK}
H_\kappa(\theta_i)=J_1(b_i)\left(1-\Frac{1}{2}\Frac{d_i}{b_i}+\Frac{d_i^2}{4 b_i^2}-\Frac{1}{2}d_i^2
+\Frac{Y_1(\beta_i)}{\kappa^2}+{\cal O}(\kappa^{-4}).
+{\cal O}(d_i^3)\right)
\end{equation}

From this, we can proceed to compute the first terms of the asymptotic expansion as $k\rightarrow +\infty$,
without any restriction on the values of $\beta_i$. We are next simplifying the expansion
by further assuming that $\beta_i=j_{n-i+1}/\kappa <<1$, which holds for the largest zeros, but no so
for the zeros close to the origin. Under this approximation $\beta_i={\cal O}(\kappa^{-1})$,
$d_i={\cal O}(\kappa^{-2})$. Additionally, using Eq. (3.12) of  \cite{Gil:2019:NIG}, in this limit
we have $Y_1 (\beta_i)=\Frac{1}{48}+{\cal O}(\kappa^{-2})$. We neglect all the terms in (\ref{HK})
except the first two and the dominant contribution from $Y_1$, yielding
$$
H_\kappa(\theta_i)=J_1(b_i)\left(1-\Frac{1}{2}\Frac{d_i}{b_i}+\Frac{1}{48\kappa^2}+{\cal O}(\kappa^{-4})\right).
$$
Then
$$
\kappa\theta_i H_k(\theta)^2 =(b_i+d_i)F_k(\theta_i)^2=b_iJ_1 (b_i)^2
\left(1+\Frac{1}{24\kappa^2}+{\cal O}(\kappa^{-4})\right),
$$
which we insert in Eq. (\ref{weBe}) to get
$$
w_{i}=\Frac{2\sin\theta_{i}}{b_i\kappa J_1 (b_i)^2}\left(1-\Frac{1}{24\kappa^2}+{\cal O}(\kappa^{-4})\right).
$$
Now, because we are considering that $\beta_i={\cal O}(\kappa^{-1})$ we have
$$
\begin{array}{ll}
\sin\theta_i&=\sin(\beta_i+\delta_i)=\sin\beta_i+\delta_i\cos\beta_i+{\cal O}(\delta_i^2)\\
& = \beta_i(1-\frac16\beta_i^2 -\frac{1}{24\kappa^2}+{\cal O}(\kappa^{-4}))
\end{array}
$$
and finally we get
\begin{equation}
w_{n-i}=\Frac{2}{\kappa^2 J_1 (j_{i+1})^2}\left(1-\Frac{1}{12\kappa^2}-
\Frac{j_{i+1}^2}{6\kappa^2}+{\cal O}(\kappa^{-4})\right).
\end{equation}

This last approximation is accurate for $i<<n$, and coincides with the expansion \cite[A.6]{Opsomer:2022:AOA}.

\section{Appendix: proof of a property of Bessel functions}
\label{Sonin}

Previously, we used the property that 
$$
a_i=1-\Frac{2}{\pi j_{i+1}J_{1}(j_{i+1})^2},\, i=0,1,\ldots
$$
is a positive decreasing sequence. We prove
a more general result, using a variant of Sonin's theorem 
(see Lemma 3.3 in \cite{Killip:2019:SAT}).

\begin{theorem}
Let $w(x)$ be a solution of $w''(x)+A(x)w(x)=0$ in some interval where $A'(x)>0$ (respectively $A'(x)<0$), then the values of $w'(x)^2$ increase (respectively 
decrease) when they are evaluated at the successive zeros of $w(x)$ (in increasing order).
\end{theorem}
\begin{proof}
Let $f(x)=w'(x)^2+A(x)w(x)^2$, which has derivative $f'(x)=A'(x)w(x)^2$ (only zero at the
zeros of $w(x)$). Then $f(x)$ has the same monotonicity as $A(x)$ and because at the 
zeros of $w(x)$ we have $f(x)=w'(x)^2$ the result is proved.
\end{proof}

Next we apply the previous result to the Bessel differential equation.

\begin{corollary}
Let $v_\nu (x)=\displaystyle\sqrt{x}{y}^{\prime}_{\nu}(x)$, with 
${y}_{\nu}(x)$ any solution of the Bessel equation ($x^2 y_{\nu}''(x)+xy_{\nu}'(x)+(x^2-\nu^2)y_{\nu}(x)=0$) 
and let $c_i$, $i=1,2,\ldots$ be the 
positive zeros of ${y}_{\nu}(x)$ in increasing order. The sequence $v_\nu(c_i)^2$, is strictly decreasing
if $|\nu|<1/2$, strictly increasing if $|\nu|>1/2$ and constant if $|\nu|=1/2$.
\end{corollary}

\begin{proof}
The function $w_{\nu}(x)=\displaystyle{\sqrt{x}}{y}_{\nu}(x)$ is solution of 
the differential equation $w''(x)+A(x)w(x)=0$ with 
$$
A(x)=1-\Frac{\nu^2-1/4}{x^2}.
$$

With this, and considering the previous theorem, we deduce that $w_{\nu}^{\prime}(c_i)^2$ increases with $i$ if
$|\nu|>1/2$ (because $A'(x)>0$), decreases if $|\nu|<1/2$ ($A'(x)>0$) and is constant if $|\nu|=1/2$ ($A'(x)=0$).
Now, because $w_{\nu}^{\prime}(c_i)=v_{\nu}(c_i)$, the monotonicity properties of the sequence  $v_\nu(c_i)^2$ are proved. 
\end{proof}

Any solution of the Bessel equations, up to a constant multiplicative factor, can be written as 
\begin{equation}
{\cal C}_{\nu}(x)=\cos\alpha J_{\nu}(x)
-\sin\alpha Y_{\nu}(x), 
\end{equation}
for some $\alpha\in [0,\pi)$,
where $J_{\nu}(x)$
and $Y_{\nu}(x)$ are the Bessel functions of the first and second kinds respectively \cite{NIST}. 
For any fixed value of $\alpha$, at the zeros $c_i$ of ${\cal C}_{\nu}(c_i)$ we have ${\cal C}'_{\nu}(c_i)=-{\cal C}_{\nu+1}(c_i)$ 
on account of \cite[10.6.2]{NIST}, and therefore
the previous Corollary implies
the next result.
\begin{corollary}
\label{coroc}
$c_i {\cal C}_{\nu+1}(c_i)^2$ is strictly increasing if $|\nu|>1/2$, strictly decreasing if $|\nu|<1/2$, and constant if $|\nu|=1/2$.
\end{corollary}

In the next discussion, we will need the asymptotic expansions for ${\cal C}_{\nu}(x)$ for large $x$ as well as the expansion for 
$c_i$ and large $i$. Combining the equations \cite[10.17,3-10.17.4]{NIST} we have
\begin{equation}
\label{AS}
{\cal C}_{\nu}(x)\sim\left(\Frac{2}{\pi x}\right)^{1/2}\left(\cos(\chi_{\nu}(x))P_{\nu}(x)-\sin(\chi_{\nu}(x))Q_{\nu}(x)\right),
\end{equation}
where
$$
\chi_{\nu}(x)=x+\alpha-\Frac{\nu \pi}{2}-\Frac{\pi}{4},
$$
$$
P_{\nu}(x)=\displaystyle\sum_{k=0}^{\infty}(-1)^k\Frac{a_{2k} (\nu)}{x^{2k}},\,
Q_{\nu}(x)=\displaystyle\sum_{k=0}^{\infty}(-1)^k\Frac{a_{2k+1} (\nu)}{x^{2k+1}},
$$
and $a_k(\nu)=(-2)^{-k}\left(\frac12-k\right)_k \left(\frac12+k\right)_k /k!$.

The dominant term in the expansion gives
$$
{\cal C}_{\nu}(x)\sim\left(\Frac{2}{\pi x}\right)^{1/2}\cos(\chi_{\nu}(x))\left(1+{\cal O}(x^{-1}))\right).
$$

McMahon asymptotic expansion for the zeros $c_i$ for large $i$ is obtained by inversion of the asymptotic series (\ref{AS}), as done 
in \cite[section 3.3.1]{Gil:2014:OTC}. The leading contribution is that which makes $\cos(\chi)=0$, that is 
$c_i+\alpha-\Frac{\nu \pi}{2}-\Frac{\pi}{4}\sim (2i-1)\Frac{\pi}{2}$, which gives $c_i\sim \left(i+\Frac{\nu}{2}-\frac14\right)-\alpha
\equiv \lambda_i$. With one additional term the expansion reads (see \cite{Gil:2014:OTC})
 \begin{equation}
 c_i=\lambda_i-\Frac{4\nu^2-1}{8\lambda_i}+{\cal O}(i^{-3}).
 \end{equation}

With the previous expansions, we can prove the following
\begin{lemma}
\label{limitc}
$$
\displaystyle\lim_{i\rightarrow \infty}\Frac{\pi c_i}{2}{\cal C}_{\nu+1}(c_i)^2=1.
$$
\end{lemma}
\begin{proof}
Considering that $\chi_{\nu}(c_i)=(2i-1)\Frac{\pi}{2}+{\cal O}(i^{-1})$ and $\chi_{\nu+1}(x)=\chi_{\nu}(x)+\Frac{\pi}{2}$ we have 
$\cos\left(\chi_{\nu +1}(c_i)\right)=(-1)^{i+1}+{\cal O}(i^{-1})$ and therefore ${\cal C}_{\nu+1}=
(-1)^{i+1} \left(\Frac{2}{\pi c_i}\right)^{1/2}(1+{\cal O}(i^{-1}))$, which proves the result
\end{proof}

As a consequence of Corollary \ref{coroc} and Lemma \ref{limitc} we can state the following:

\begin{corollary}
The sequence $\{a_i\}$, $a_i=1-\Frac{2}{\pi c_i {\cal C}_{\nu+1}(c_i)^2}$, $0<c_1<c_2<\cdots$ being the positive zeros of ${\cal C}_{\nu}(x)$, is
strictly decreasing and positive if $|\nu|<1/2$, strictly increasing and negative if $|\nu|<1/2$ and
constant if $|\nu|=1/2$. Additionally $\displaystyle\lim_{i\rightarrow \infty}s_i=0$.
\end{corollary}

The result that we have mentioned at the beginning of this Appendix is simply the case $\alpha=0$ of the previous corollary.

\section{Appendix: proof that the limit $C_{+\infty}$ is finite}

\label{apc}

We prove the convergence in a more general case. Let us define
$$
C_i=2\displaystyle\sum_{k=0}^{i-1}J_\nu^{\prime} (j_{k+1})^{-2}+J_\nu^{\prime} (j_{i+1})^{-2}-\Frac{j_{i+1}^2}{2}
$$
with $j_i$ the $i$-th positive zero of $J_{\nu}(x)$. For the case $\nu=0$ we obtain our previous definition.

 We write
 $$
 C_{i-1}=C_0+\displaystyle\sum_{k=1}^{i-1}s_k,\,s_k=C_{k}-C_{k-1}.
 $$
 We have
 $$
 s_k=J_{\nu}^{\prime}(j_{k+1})^{-2}+J_{\nu}^{\prime}(j_{k})^{-2}-\Frac{j_{k+1}^2-j_k^2}{2}.
 $$
 We are checking that as $k\rightarrow +\infty$, $s_k={\cal O}(k^{-2})$, which means that $\displaystyle\sum_{k=1}^{+\infty} s_k <+\infty$
 and completes the proof.

 For this purpose, using \cite[10.17.9]{NIST}, we have
 $$
 J_{\nu}^{\prime}(j_{k})^{-2}=\Frac{\pi}{2}j_{k}Q_{\nu}(j_{k})^{-2},
 $$
 where $Q_{\nu}(z)$ admits, as $z\rightarrow +\infty$ the expansion
 $$
 Q_{\nu}(z)\sim\sin w\displaystyle\sum_{i=0}^{+\infty}(-1)^i\Frac{b_{2i}}{z^{2i}}
 + \cos w\displaystyle\sum_{i=0}^{+\infty}(-1)^i\Frac{b_{2i+1}}{z^{2i+1}},
 $$
 where $w=z-\Frac{\nu\pi}{2}-\Frac{\pi}{4}$. Now, for large $k$ we have \cite[10.21.19]{NIST}
 \begin{equation}
 \label{zerola}
 j_{k}=a_k+\delta_k,\,\delta_k = \Frac{1-4\nu^2}{8a_k}+{\cal O}(a_k^{-3})
 ,\,a_k=\left( k +\Frac{\nu}{2}-\Frac{1}{4}\right)\pi.
 \end{equation}

 Now setting $z=j_k$, $w=\left(k-\frac12\right)\pi+\delta_k$ and then $\sin w=(-1)^{k+1}\cos\delta_k$,
  $\cos w=(-1)^k\sin\delta_k$. With this we get $Q_{\nu}(j_k)=1+{\cal O}(k^{-2})$ and
  $$
  s_k=\Frac{\pi}{2}\left(j_k + j_{k+1}\right)(1+{\cal O}(k^{-2}))
  -\Frac{1}{2}(j_{k+1}-j_k)(j_{k+1}+j_k)
  $$

  Now, because of (\ref{zerola})
  $$
  j_{k+1}-j_k\sim \pi +\Frac{(4\nu^2-1)\pi}{8 a_{k+1} a_{k} }=\pi+{\cal O}(k^{-2}).
  $$

  With this $s_k={\cal O}(k^{-2})$.

\section{Appendix: proof of Observation \ref{goal}}
\label{ApE}
In this appendix $x_i$ will be a generic point of the mesh, rather than the $i^{\rm th}$ Gaussian node. By using standard analysis of numerical schemes for initial value problems (IVPs), we will prove convergence of order 2 for a scheme which contains (\ref{mid-point-2}) as a particular case.

\par Consider, to be solved with some numerical scheme, the scalar IVP
\begin{equation}
\label{IVP-Appendix}
\left\{
\begin{array}{l}
y' = f(x)\ \mbox{in}\ (-1,1),\\
y(-1) = \eta \in \RE,
\end{array}
\right.
\end{equation}
where $f \in {\rm C}^2([-1,1])$. This problem has a unique solution $y \in {\rm C}^3([-1,1])$.

\par We are going to employ meshes
\begin{equation}
\label{mesh}
x_0 = -1 < x_1 < \dots < x_{n-1} < x_n = 1
\end{equation}
such that, if $h_i = x_{i+1} - x_i$ and $H_n = \max_{0\leq i \leq n-1} h_i$, one has
\begin{equation}
\label{hypothesis-size-1}
H_n = {\cal O}(n^{-1}).
\end{equation}

\par On any of these meshes, define the following scheme:
\begin{equation}
\label{scheme}
\left\{
\begin{array}{l}
y_0 = \eta,\\
y_{i+1} = y_i + h_i f(x_i + h_i^*/2)\ \mbox{for}\ i = 0,\dots,n-1,
\end{array}
\right.
\end{equation}
where $h_i^* = h_i + d_i$, with $D_n = \max_{0\leq i \leq n-1} |d_i|$ satisfying
\begin{equation}
\label{hypothesis-size-2}
D_n = {\cal O}(n^{-2}).
\end{equation}

\par When $h_i = h_i^* = h = 2/n$, the scheme (\ref{scheme}) is easily recognizable as more than one numerical method applied to problem (\ref{IVP-Appendix}). One is the mid\=/point rule adapted to have only one step (this can be done because $f$ does not depend on $y$); another one is the modified Euler scheme, which can be seen as a member of the Runge\=/Kutta family with the following Butcher tableau:
\begin{equation}
\begin{array}{c|cc}
 0 & 0 & 0 \\
 1/2 & 1/2 & 0 \\
\hline \\[-3mm]
  & 0 & 1.
\end{array}
\end{equation}

\par Both the mid\=/point rule and the modified Euler scheme are known to have order 2, so the following result is quite natural.

\begin{theorem}
\label{thm-convergence}
Under the hypotheses (\ref{hypothesis-size-1}) and (\ref{hypothesis-size-2}), the scheme (\ref{scheme}) is convergent of order 2, that is to say, $\max_{0 \leq i \leq n} |y(x_i) - y_i| = {\cal O}(n^{-2})$.
\end{theorem}

\par The proof of Theorem \ref{thm-convergence} follows the standard approach, which is proving convergence by means of consistency and stability, but, since in this simple case function $f$ does not depend on $y$, stability is trivially satisfied.

\begin{lemma}\label{lemma-consistency}
Under the hypotheses (\ref{hypothesis-size-1}) and (\ref{hypothesis-size-2}), the scheme (\ref{scheme}) is consistent of order 2, that is to say, for any solution $y$ of $y' = f(x)$,
\begin{equation}
\max_{1 \leq i \leq n} |\tau_i| = {\cal O}(n^{-2}),
\end{equation}
where $\tau_{i+1} = (y(x_{i+1}) - y(x_i))/h_i - f(x_i + h_i^*/2)$ for $i = 0,...,n-1$.
\end{lemma}

\begin{proof}
In this proof we will employ the notation $\|\varphi\|_\infty = \max_{x \in [-1,1]} |\varphi(x)|$ for $\varphi \in {\rm C}([-1,1])$.

\par Notice that $\tau_{i+1} = (y(x_{i+1}) - y(x_i))/h_i - y'(x_i + h_i^*/2)$ because $y$ is a solution of $y' = f(x)$, and recall that $y \in {\rm C}^3([-1,1])$. Now use the Taylor expansions
\begin{equation}
y(x_{i+1}) = y(x_i) + h_i y'(x_i) + \frac{h_i^2}{2} y''(x_i) + \frac{h_i^3}{6} y'''(\xi_i)
\end{equation}
and
\begin{equation}
y'(x_i + h_i^*/2) = y'(x_i) + \frac{h_i^*}{2} y''(x_i) + \frac{(h_i^*)^2}{8} y'''(\xi_i^*)
\end{equation}
together with $h_i^* = h_i + d_i$ to write
\begin{eqnarray}
\nonumber \tau_{i+1} = \{y'(x_i) + \frac{h_i}{2} y''(x_i) + \frac{h_i^2}{6} y'''(\xi_i)\} -\\
\nonumber \{y'(x_i) + \frac{(h_i + d_i)}{2} y''(x_i) + \frac{(h_i + d_i)^2}{8} y'''(\xi_i^*)\} = \\
\frac{h_i^2}{6} y'''(\xi_i) - \frac{d_i}{2} y''(x_i) - \frac{(h_i + d_i)^2}{8} y'''(\xi_i^*).
\end{eqnarray}

\par So,
\begin{equation}
|\tau_{i+1}| \leq \frac{H_n^2}{6} \|y'''\|_\infty + \frac{D_n}{2} \|y''\|_\infty + \frac{(H_n^2 + 2 H_n D_n + D_n^2)}{8} \|y'''\|_\infty
\end{equation}
for $i = 0,...,n-1$, and, as the bound on the right hand side does not depend on $i$, the result follows from the hypotheses (\ref{hypothesis-size-1}) and (\ref{hypothesis-size-2}).
\end{proof}

\par Now Theorem \ref{thm-convergence} can be proved as follows.

\begin{proof}[Proof of Theorem \ref{thm-convergence}]
The conclusion $\max_{0 \leq i \leq n} |y(x_i) - y_i| = {\cal O}(n^{-2})$ will be proved if $\max_{1 \leq i \leq n} |y(x_i) - y_i| = {\cal O}(n^{-2})$, as $y(x_0) = y_0 = \eta$.

\par Due to the definition of $\tau_{i+1}$ in Lemma \ref{lemma-consistency}, the values $y(x_i)$ of the exact solution satisfy
\begin{equation}
\left\{
\begin{array}{l}
y(x_0) = \eta,\\
y(x_{i+1}) = y(x_i) + h_i (f(x_i + h_i^*/2) + \tau_{i+1})\ \mbox{for}\ i = 0,\dots,n-1,
\end{array}
\right.
\end{equation}
while the approximate values $y_i$ satisfy the scheme (\ref{scheme}). Hence, we have that, for $i = 0,\dots,n-1$,
\begin{equation}
y(x_{i+1}) - y_{i+1} = y(x_i) - y_i + h_i \tau_{i+1},
\end{equation}
so
\begin{equation}
|y(x_{i+1}) - y_{i+1}| \leq |y(x_i) - y_i| + H_n |\tau_{i+1}|.
\end{equation}

\par By induction,
\begin{equation}
|y(x_i) - y_i| \leq H_n \sum_{j=1}^i |\tau_j|
\end{equation}
for $i = 1,...,n$, which implies
\begin{equation}
\label{eq-apoyo}
\max_{1 \leq i \leq n} |y(x_i) - y_i| \leq H_n \sum_{i=1}^n |\tau_i|.
\end{equation}

\par The proof ends by combining (\ref{eq-apoyo}) with the inequality
\begin{equation}
H_n \sum_{i=1}^n |\tau_i| \leq n H_n \max_{1 \leq i \leq n} |\tau_i|,
\end{equation}
the hypothesis (\ref{hypothesis-size-1}), and the consistency of order 2 stated in Lemma \ref{lemma-consistency}.
\end{proof}

Notice that, under Conjecture \ref{mainc1}, the scheme (\ref{mid-point-2}) converges with order 2 in virtue of Theorem \ref{thm-convergence}. Indeed, the collection $\{h_i\}_{i=0}^{n-1} = \{w_i\}_{i=1}^n$ is known to satisfy (\ref{hypothesis-size-1}), and Conjecture \ref{mainc1} implies that, for the $h_i^*$ in (\ref{mid-point-2}), hypothesis (\ref{hypothesis-size-2}) holds as well. So, we have in this way a detailed proof of Observation \ref{goal}.

 \subsection*{Acknowledgements}
 The authors thank the reviewers for many detailed and helpful comments. JS thanks A. Gil and N. M. Temme for useful comments.
 
 \subsection*{Author contributions}
 
 Both authors contributed equally to this work

 \subsection*{Funding}

LP acknowledges support from Ministerio de Ciencia e Innovaci\'on, project PID2021-122625OB-I00 with funds from MCIN/AEI/10.13039/501100011033/ FEDER, UE, and from the Xunta de Galicia (2021 GRC Gl-1563 - ED431C 2021/15). JS acknowledges support from Ministerio de Ciencia e Innovaci\'on, project
 PID2021-127252NB-I00 with funds from MCIN/AEI/10.13039/501100011033/ FEDER, UE.
 
 \subsection*{Availability of data and materials}
 
 Not applicable.
 
 \section*{Declarations}
 
 \subsection*{Competing interests}
 
 The authors declare no competing interests.


\begin{thebibliography}{10}

\bibitem{Bogaert:2014:IFC}
I.~Bogaert.
\newblock Iteration-free computation of {G}auss-{L}egendre quadrature nodes and
  weights.
\newblock {\em SIAM J. Sci. Comput.}, 36(3):A1008--A1026, 2014.

\bibitem{Davis:1961:SGT}
P.~J. Davis and P.~Rabinowitz.
\newblock Some geometrical theorems for abscissas and weights of {G}auss type.
\newblock {\em J. Math. Anal. Appl.}, 2:428--437, 1961.

\bibitem{Frenzen:1985:AUA}
C.~L. Frenzen and R.~Wong.
\newblock A uniform asymptotic expansion of the {J}acobi polynomials with error
  bounds.
\newblock {\em Canad. J. Math.}, 37(5):979--1007, 1985.

\bibitem{Gattes:1985:AAE}
Luigi Gatteschi and Giovanna Pittaluga.
\newblock An asymptotic expansion for the zeros of {J}acobi polynomials.
\newblock In {\em Mathematical analysis}, volume~79 of {\em Teubner-Texte
  Math.}, pages 70--86. Teubner, Leipzig, 1985.

\bibitem{Gautschi:2006:TCT}
Walter Gautschi.
\newblock The circle theorem and related theorems for {G}auss-type quadrature
  rules.
\newblock {\em Electron. Trans. Numer. Anal.}, 25:129--137, 2006.

\bibitem{Gil:2014:OTC}
Amparo Gil and Javier Segura.
\newblock On the complex zeros of {A}iry and {B}essel functions and those of
  their derivatives.
\newblock {\em Anal. Appl. (Singap.)}, 12(5):537--561, 2014.

\bibitem{Gil:2019:NIG}
Amparo Gil, Javier Segura, and Nico~M. Temme.
\newblock Noniterative computation of {G}auss-{J}acobi quadrature.
\newblock {\em SIAM J. Sci. Comput.}, 41(1):A668--A693, 2019.

\bibitem{Gil:2021:FAR}
Amparo Gil, Javier Segura, and Nico~M. Temme.
\newblock Fast and reliable high-accuracy computation of {G}auss-{J}acobi
  quadrature.
\newblock {\em Numer. Algorithms}, 87(4):1391--1419, 2021.

\bibitem{HA-LI}
Pierre\=/Andr\'e Haldy and Jacques Ligou.
\newblock A multigroup formalism to solve the {F}okker\=/{P}lanck equation
  characterizing charged particle transport.
\newblock {\em Nucl. Sci. Eng.}, 74(3):178--184, 1980.

\bibitem{Killip:2019:SAT}
Rowan Killip and Monica Visan.
\newblock Sonin's argument, the shape of solitons, and the most stably singular
  matrix.
\newblock In {\em Harmonic analysis and nonlinear partial differential
  equations}, RIMS K\^{o}ky\^{u}roku Bessatsu, B74, pages 23--32. Res. Inst.
  Math. Sci. (RIMS), Kyoto, 2019.

\bibitem{Laurie:2001:COG}
Dirk~P. Laurie.
\newblock Computation of {G}auss-type quadrature formulas.
\newblock volume 127, pages 201--217. 2001.
\newblock Numerical analysis 2000, Vol. V, Quadrature and orthogonal
  polynomials.

\bibitem{Lopez:2023:AOD}
Oscar L\'opez~Pouso and Javier Segura.
\newblock Analysis of difference schemes for the {F}okker-{P}lanck angular
  diffusion operator.
\newblock {\em Comput. Math. Appl.}. 181:84-110, 2025.

\bibitem{MO}
Jim~E. Morel.
\newblock An improved {F}okker\=/{P}lanck angular differencing scheme.
\newblock {\em Nucl. Sci. Eng.}, 89(2):131--136, 1985.

\bibitem{OL-FR}
Edgar Olbrant and Martin Frank.
\newblock Generalized {F}okker\=/{P}lanck theory for electron and photon
  transport in biological tissues: application to radiotherapy.
\newblock {\em Comput. Math. Methods Med.}, 11(4):313--339, 2010.

\bibitem{NIST}
Frank W.~J. Olver, Daniel~W. Lozier, Ronald~F. Boisvert, and Charles~W. Clark,
  editors.
\newblock {\em N{IST} handbook of mathematical functions}.
\newblock U.S. Department of Commerce, National Institute of Standards and
  Technology, Washington, DC; Cambridge University Press, Cambridge, 2010.
\newblock With 1 CD-ROM (Windows, Macintosh and UNIX).

\bibitem{Opsomer:2022:AOA}
Peter Opsomer and Daan Huybrechs.
\newblock High-order asymptotic expansions of gaussian quadrature rules with
  classical and generalized weight functions.
\newblock {\em J. Comput. Appl. Math.}, 434:115317, 2023.

\bibitem{PA-WA-PR}
Japan~K. Patel, James~S. Warsa, and Anil~Kant Prinja.
\newblock Accelerating the solution of the ${S}_{N}$ equations with highly
  anisotropic scattering using the {F}okker\=/{P}lanck approximation.
\newblock {\em Annals Nucl. Eng.}, 147:107665, 2020.

\bibitem{WA-PR}
James~S. Warsa and Anil~Kant Prinja.
\newblock A moment preserving ${S}_{N}$ discretization for one\=/dimensional
  {F}okker\=/{P}lanck equation.
\newblock {\em Trans. Am. Nucl. Soc.}, 106:362--365, 2012.

\end{thebibliography}
\end{document}